\documentclass[a4paper]{article}
\usepackage[utf8]{inputenc}
\usepackage[T1]{fontenc}
\usepackage{graphicx}
\usepackage{hyperref}

\usepackage{amsmath}

\usepackage{amsthm}

\usepackage{amssymb}
\usepackage{mathabx}

\usepackage[all]{xy}

\usepackage[english]{babel}

\newcommand\strong[1]{{\bfseries #1}}

\DeclareMathSymbol{\theta}{\mathord}{letters}{"23}
\DeclareMathSymbol{\vartheta}{\mathord}{letters}{"12}
\DeclareMathSymbol{\phi}{\mathord}{letters}{"27}
\DeclareMathSymbol{\varphi}{\mathord}{letters}{"1E}

\newcommand\CC{\mathbb{C}}

\newcommand\NN{\mathbb{N}}

\newcommand\Db{\mathcal{D}b}
\newcommand\dbar{\bar\partial}
\newcommand\dd{\partial}
\newcommand\ra{\rightarrow\nobreak}

\newcommand\ab{\hbox{$(a,b)$-mod}\-ule}
\newcommand\CCb{\mathbb{C}[[b]]}
\newcommand\Ead{\breve{E}^*} 
\newcommand\otimesab{\otimes_{(a,b)}}

\DeclareMathOperator{\diff}{d}
\DeclareMathOperator{\dz}{dz}
\DeclareMathOperator{\df}{df}
\DeclareMathOperator{\Ker}{Ker}
\DeclareMathOperator{\im}{Im}

\DeclareMathOperator{\Hom}{Hom}

\newtheorem{theorem}{Theorem}
\newtheorem{proposition}[theorem]{Proposition}
\newtheorem{lemma}[theorem]{Lemma}
\newtheorem{definition}[theorem]{Definition}
\newtheorem{corollary}[theorem]{Corollary}

\theoremstyle{definition}
\newtheorem{example}[theorem]{Example}

\theoremstyle{remark}
\newtheorem{remark}[theorem]{Remark}

\title{Hermitian $(a,b)$-modules and Saito's \hbox{``higher residue pairings''}}
\author{Piotr P. Karwasz\thanks{Institut Élie Cartan Nancy UMR 7502\newline
Nancy-Université, B.P. 70230, 54506 Vandœuvre-lès-Nancy, France\newline
E-mail: \texttt{piotr.karwasz@iecn.u-nancy.fr}}}
\begin{document}
\maketitle
\begin{abstract}
  Following the work of Daniel~\textsc{Barlet} (\cite{barlet2}) and Ridha~\textsc{Belgrade} (\cite{belgrade}) the aim of
  this article is the study of the existence of $(a,b)$-hermitian forms on regular {\ab}s. We show that every regular
  {\ab} $E$ with a non-degenerate bilinear form can be written in an unique way as a direct sum of {\ab}s $E_i$ that
  admit either an $(a,b)$-hermitian or an $(a,b)$-anti-hermitian form or both; all three cases are equally possible with
  explicit examples.
  
  As an application we extend the result in \cite{belgrade} on the existence for all {\ab}s $E$ associated with the
  Brieskorn module of a holomorphic function with an isolated singularity, of an $(a,b)$-bilinear non degenerate form on
  $E$. We show that with a small transformation Belgrade's form can be considered $(a,b)$-hermitian and that the result
  satisfies the axioms of Kyoji Saito's ``higher residue pairings''.

\end{abstract}

\section{Introduction}

In this article we will study the self-duality properties of {\ab}s and more precisely the conditions under which they admit
a nondegenerate hermitian form. As such we wish to provide the reader with a short introduction to the theory of {\ab}s.

The {\ab}s were introduced by D.~\textsc{Barlet} in \cite{barlet1} as a formal completion of the Brieskorn module
(\cite{brieskorn}) \[
D:=\frac{\Omega_0^{n+1}}{\df\wedge\diff \Omega_0^{n-1}} \] associated to a holomorphic function $f:\CC^{n+1}\ra\CC$ with
an isolated singularity at the origin, where we denote by $\Omega_0^p$ the germs of holomorphic $p$-forms in $0$.

We wish to recall recall briefly the basic results about {\ab}s and refer the reader to the articles \cite{barlet1} and
\cite{barlet2} for further details.

\begin{definition}
  \label{def:ab-module}
  Let $\CC[[b]]$ be the ring of formal series in the variable $b$.  An \textbf{\hbox{$\mathbf{(a,b)}$-mod}\-ule} is an
  algebraic structure composed by a free {$\CC[[b]]$-mod}\-ule $E$ of finite rank and a {$\CC$-li}\-near application
  $a:E\ra E$ that satisfies the commutation relation
  \begin{equation}
    \label{eqn:abba}
    ab - ba = b^2,
  \end{equation}
  where $b:E\ra E$ is the multiplication by the element $b\in\CC[[b]]$.
\end{definition}

For a complex number $\lambda\in \CC$ and an {\ab} $E$, we call \textbf{monomial of type $(\lambda,0)$}, an element
$x\in E$ that satifies to the relation $ax=\lambda bx$. The simplest {\ab}s are those generated over $\CCb$ by a
monomial $e_{\lambda}$ of type $(\lambda,0)$. These modules are called \textbf{elementary} and noted $E_\lambda$.

Given an {\ab} $E$, a sub-$\CCb$-module $F$ of $E$ closed to the multiplication by $a$ is called sub-$(a,b)$-modules
Since the quotient of an {\ab} $E$ by a sub-$(a,b)$-module $F$ is not necessarily $b$-torsion free, a sub-$(a,b)$-module
$F$ of $E$ will be called \textbf{normal}, if $E/F$ is free and hence have an induced {\ab} structure.

The {\ab}s associated to a Brieskorn module are all \strong{regular}, i.e. they are sub-$(a,b)$-modules an {\ab} $E$
satisfying $aE\subset bE$ (a \strong{simple-pole} {\ab}). The composition series of regular {\ab} satisfy the following
property:
\begin{proposition}
  Let $E$ be a regular {\ab}, then all its composition series are of the form:
  \[
  0 = E_0 \subsetneq \dots \subsetneq E_{n-1} \subsetneq E_n = E,
  \]
  with $E_i/E_{i-1}$ elementary {\ab}s $E_\lambda$.
\end{proposition}

As proven in \cite{barlet1}, the quotients of two composition series of an {\ab} $E$ are not unique, even if we ignore
consider the permutations of the quotients.

\section{The {\ab}s and their duality}

The dual and bi-dual structure on {\ab}s where first introduced in \cite{barlet2} and \cite{belgrade} and then expanded
in our thesis (cf. \cite{mythesis}).  We will therefore begin by giving a formal definition of the duality structures we
work with.

In the spirit of the category theory we will define an \textbf{\hbox{$\mathbf{(a,b)}$-mor}\-phism} as an application
$\phi: E \rightarrow F$ between two {\ab}s $E$ and $F$, which is a morphism of the underlying \hbox{$\CCb$-mod}\-ules
and respects the \hbox{$a$-struc}\-ture $\phi(ax) = a\phi(x)$, for any element $x\in E$. We will call $\phi$ an
\textbf{isomorphism} (resp.  \textbf{endomorphism}) of {\ab}s if it is bijective (resp. $E$ = $F$).

\subsection{$(a,b)$-linear maps and dual {\ab}s}

Let $E$ and $F$ be two {\ab}s. As defined by D.~Barlet in \cite{barlet2}, the \hbox{$\CCb$-mod}\-ule $\Hom_{\CCb}\left(
E,F \right)$ of $\CCb$-linear maps from $E$ to $F$ has a natural structure of {\ab} provided by an operator $\Lambda$
that satisfies
\begin{equation}
  \label{eq:def_hom} \left( \Lambda \phi \right) (x) = a_F\left( \phi(x) \right) - \phi\left( a_Ex \right),
\end{equation}
where $\phi\in\Hom_{\CCb}\left( E,F \right)$, $x$ is an element of $E$ and $a_E$ and $a_F$ are the $a$-structures of $E$
and $F$ respectively. We use for this {\ab} the notation $\Hom_{(a,b)}\left( E,F \right)$.

For notation's sake we will denote $a_E$, $a_F$ and $\Lambda$ all by the letter $a$ and to avoid the confusion that such
a notation could pose we should read the expression
\[
a\cdot \phi(x)
\]
as $\left( \Lambda \phi \right)(x)$, whereas the expression $a_E\left( \phi(x) \right)$ will keep the conventional
notation
\[
a\phi(x).
\]
We will therefore rewrite the equation \ref{eq:def_hom} as: $a\cdot \phi(x) = a\phi(x) - \phi(ax)$.

By choosing $E_0$ for the codomain of the morphisms, we can give the following definition:

\begin{definition}[Barlet]
  \label{def:dual_ab-module} Let $E$ be an {\ab} and $E_0$ the elementary {\ab} of parameter~$0$, then we call the
  module
  \[
  \Hom_{(a,b)}\left( E, E_0 \right)
  \]
  the \textbf{dual {\ab}} of $E$ and note it by $E^*$.
\end{definition}

\begin{remark}
  When considering only the $b$-structure of~$E$, the $\CCb$-module $E^*$ corresponds exactly to the definition of dual
  of a $\CCb$-module, since $E_0= \CCb e_0$, with $ae_0 =0$.
\end{remark}
 
The duality functor ${}^*$ is exact (cf. \cite{barlet2}).

\subsection{Conjugate {\ab}}

In~\cite{belgrade} R.~\textsc{Belgrade} uses another definition of dual {\ab} which is not equivalent to the one of
D.~\textsc{Barlet}. In order to be able to express on concept in terms of the other the other, we will introduce an
operation that exchanges the signs of both $a$ and $b$, whose behaviour is similar to that of conjugation in the complex
field.

As in the case of the complex field $\CC$, the ring of formal series $\CCb$ also admits a rather natural involution
\begin{align*}
  \breve{} : \CCb &\ra \CCb\\
  S(b) &\mapsto \breve{S}(b) = S(-b),
\end{align*}
where $S(b)\in\CCb$. This remark allows us to define the conjugate of an {\ab} in the same way as one defines the
conjugate of a complex vector space.

\begin{definition}
  Let $E$ be an {\ab}. We call \textbf{$(a,b)$-conjugate} of~$E$ and note it $\breve{E}$ the set $E$ itself, endowed
  with an $a$- and~$b$-structure given by:
  \begin{align*}
  a\cdot_{\breve{E}}v &= -a\cdot_Ev\\
  b\cdot_{\breve{E}}v &= -b\cdot_Ev,
  \end{align*}
  where $\cdot_{\breve{E}}$ and $\cdot_E$ denote the $(a,b)$-structure of $\breve{E}$ and~$E$ respectively.

  Since we change signs of both $a$ and~$b$, the formula $ab - ba = b^2$ is still satisfied.
\end{definition}

\begin{remark}
  An {\ab} is not necessarily isomorphic to its conjugate. We can take, for example, the {\ab} of rank~$2$ generated by
  two elements $x$ and $y$ that satisfy:
  \begin{align*}
  ax & = \lambda bx\\
  ay & = \lambda by + \left( 1 + \alpha b \right)x,
  \end{align*}
  where $\lambda$ and $\alpha\in\CC$ and $\alpha\neq 0$. Its conjugate satisfies
  \begin{align*}
    ax & = \lambda bx\\
    ay & = \lambda by + \left( 1 - \alpha b\right)x,
  \end{align*}
  and the classification of rank~$2$ regular {\ab}s, given in~\cite{barlet1} implies that the two modules are not
  isomorphic.
\end{remark}

One can see immediately that for an {\ab} $E$ the conjugate of the conjugate $\left( \breve{E} \right)\breve{}$ is the
{\ab} itself.

On the other hand let $E$ and~$F$ be {\ab}s and $\phi$ a morphism between $E$ and~$F$. Since $\phi(-ax)=-a\phi(x)$ and
$\phi(-bx)= -b\phi(x)$, for all $x\in E$ the application $\phi$ is also a morphism between the conjugates $\breve{E}$
and $\breve{F}$. We call conjugation functor the functor that associates to every {\ab} its conjugate and to every
morphism, the morphism itself. Such a functor is exact.

For an {\ab} module $E$ we will be interested especially in a particular kind of conjugate, the conjugate of the dual,
which we call \textbf{adjoint} {\ab} and note with $\Ead$.

\subsection{Bilinear forms and tensor product}

In order to define $\Hom_{(a,b)}\left( E,F \right)$ we used the equivalent object for its underlying $b$-structure. We
can proceed in a similar way to obtain the concept of $(a,b)$-bilinear maps:
\begin{definition}
  Let $E$, $F$ and~$G$ be two {\ab}s. An \textbf{$(a,b)$-bilinear map} on $E\times F$ is a $\CCb$-linear map $\Phi$,
  \[
  \Phi: E\times F \ra G,
  \]
  that satisfies the following property:
  \[
  a\Phi(x,y) = \Phi(ax,y) + \Phi(x,ay).
  \]
\end{definition}

\begin{remark}
  \label{rmk:bilinear_linear}
  If $\Phi$ is an $(a,b)$-bilinear map on $E\times F$ with values in $G$ and $v$ is an element of $E$:
  \[
  \Phi_v:= \Phi(v,\cdot): w \mapsto \Phi(v,w)\qquad w\in F
  \]
  is not necessarily an $(a,b)$-morphism. However the map $\pi:v\mapsto \Phi_v$ is an $(a,b)$-morphism between $E$
  and~$\Hom_{(a,b)}\left( F, G \right)$.  We have in fact:
  \[
  \pi(av) (x) = \Phi_{av} (x) = a\Phi_v (x) - \Phi_v (ax) = a\cdot\Phi_v(x)
  = a\pi(v).\]
\end{remark}

Inherently linked to the concept of $(a,b)$-bilinear maps is that of tensor products, that allows a more practical
manipulation of these objects.

\begin{definition}
  Let $E$ and~$F$ be two {\ab}s. We call $(a,b)$-tensor product of $E$
  and~$F$ and write it as $E\otimes_{(a,b)}F$ the $\CCb$-module
  \[
  E\otimes_{\CCb} F
  \]
  endowed with an $a$-structure defined as follows:
  \[
  a\left( v\otimes w \right) = \left( av \right)\otimes w + v\otimes \left(
  aw \right)
  \]
  for every $v\in E$ and $w\in F$.
\end{definition}

The $a$-structure we gave on $E\otimesab F$ is well defined. We have in
fact:
\begin{multline*}
  a\left( bv \otimes w \right)= a(bv) \otimes w + bv \otimes a(w) = ba(v)
  \otimes w + b^2v \otimes w + v \otimes ba(w) =\\
  a(v) \otimes bw + v\otimes a(bw) = a\left( v\otimes bw \right),
\end{multline*}
for each $v\in E$, $w\in F$ and it satisfies $ab - ba = b^2$:
\begin{multline*}
  a\left( bv \otimes w \right) - ba\left( v\otimes w \right) = ba(v)
  \otimes w + b^2v \otimes w + bv \otimes a(w) \\
  - ba(v) \otimes w -
  bv\otimes a(w) = b^2\left( v\otimes w \right).
\end{multline*}

We can easily verify that the tensor product defined satisfies the usual universal property: there exists a bilinear map
\[
\Phi: E \times F \ra E \otimesab F,
\]
such that for every bilinear map $\Psi$ on $E \times F$ with values in a third {\ab} $G$, there exists an unique
$(a,b)$-morphism $\tilde\Psi$ from $E \otimesab F$ into $G$ that makes the following diagram commutative:
\[\xymatrix{
E \times F \ar[d]^\Phi \ar[r]^\Psi & G\\
E \otimesab F \ar@{-->}[ur]^{\tilde\Psi}.
}\]
We can take as $\Phi$ the natural application
\begin{align*}
  \Phi: E \times F &\ra E \otimesab F\\
  (v,w) &\mapsto v \otimesab w
\end{align*}
and define $\tilde\Psi$ as:
\begin{align*}
  \tilde\Psi: E \otimesab F &\ra G\\
  v \otimesab w & \mapsto \Psi(v,w)
\end{align*}
The unicity of $\tilde\Psi$ follows directly from the universal property of the tensor product of $\CCb$-modules. We
need only to verify that the map is $a$-linear. We will do it on the generators $v \otimesab w$ of $E \otimesab F$, for
$v \in E$ and $w \in F$:
\begin{multline*}
\tilde\Psi\left( a(v \otimesab w) \right) = \tilde\Psi\left( (av) \otimesab
w + v \otimesab (aw) \right) = \\
\Psi(av,w) + \Psi(v, aw) = a\Psi(v,w) =
a\tilde\Psi(v \otimesab w).
\end{multline*}

Basing ourselves on the properties of the tensor product of $\CCb$-modules, we can derive in a similar manner the other
properties of the equivalent object in the theory {\ab}s.

\begin{lemma}
  \label{lem:product_prop}
  Let $E$, $F$ and~$G$ be three {\ab}s, then the tensor product verifies the following properties:
  \begin{enumerate}
  \item \[E\otimes_{(a,b)} F \simeq F \otimes_{(a,b)} E,\]
  \item \[\left( E \otimes_{(a,b)} F
    \right) \otimes_{(a,b)} G \simeq E \otimes_{(a,b)} \left( F
    \otimes_{(a,b)} G \right),\]
  \item \[\left( E\otimesab F \right)^* \simeq E^* \otimesab F^*,\]
  \item \[\left( E\otimesab F \right)\breve{} \simeq \breve{E} \otimesab
    \breve{F},\]
  \item The $(a,b)$-morphism
    \begin{align*}
      \Phi:E &\ra E \otimesab E_0\\
      v &\mapsto v\otimesab e_0
    \end{align*}
    where $e_0$ is a generator of the elementary {\ab} $E_0$, is an
    isomorphism.
  \item \label{tensor_hom}
    We have the following isomorphism of {\ab}s:
    \begin{align*}
    E^* \otimesab F &\ra \Hom_{(a,b)}\left( E, F \otimesab E_0\right)\\
    \phi \otimesab y &\mapsto \left( \Phi: x \mapsto  y \otimesab
    \phi(x)\right),
    \end{align*}
    where $\phi\in E^*$, $x\in E$ and~$y\in F$.
\end{enumerate}
\end{lemma}

\begin{remark}
  In \cite{belgrade}, R.~\textsc{Belgrade} defines the concept of \textbf{$\delta$-dual} of an {\ab} $E$:

  \begin{definition}[Belgrade] Let $E$ be an {\ab} and $\delta\in\CC$, then we call the $\delta$-dual of $E$ the set
    $\Hom_{\CCb}(E,E_\delta)$ with the $(a,b)$-structure defined as follow:
    \begin{align}
      [a\cdot \phi](x) &= \phi(ax) - a\phi(x)\\
      [b\cdot \phi](x) &= -b\phi(x) = \phi(-bx)
    \end{align}
    \label{def:belgrade}
  \end{definition}

  From property (v) and (vi) of the previous lemma we obtain the isomorphism $E^*\otimesab F \simeq \Hom_{(a,b)}\left(
  E,F \right)$, which in turn let us find an alternative description of the $\delta$-dual of an {\ab}. In fact from
  definition~\ref{def:belgrade} it is easy to show that the $\delta$-dual of an {\ab} in Belgrade's terminology is the
  module
  \[
  \Hom_{(a,b)}\left( \check{E}, E_\delta \right),
  \]
  which in turn can be rewritten as $\Ead \otimesab E_\delta$.
\end{remark}

We will call an $(a,b)$-bilinear application on $E \times F$ with values in $G$, an \textbf{$(a,b)$-bilinear form} if
$G=E_0$. In the rest of this section we will deal with the existence of nondegenerate hermitian forms on {\ab}s. We will
need therefore the following definitions.

\begin{definition}
  Let $E$ and~$F$ be two {\ab}s and $\Phi$ a bilinear form on $E\times F$.  We say that $\Phi$ is
  \strong{nondegenerate}, if the $(a,b)$-morphism $v\mapsto \Phi(v,\cdot)$ is an isomorphism of $E$ with $F^*$.
\end{definition}

\begin{definition}
  Let $E$ be an {\ab}. A \strong{sesquilinear} form on $E$ is a bilinear form on $E \times \check{E}$.
\end{definition}

\begin{remark}
  \label{rmk:forcibly_self-adjoint}
  Since a nondegenerate sequilinear form on an {\ab} $E$ induces an isomorphism to its adjoint $\Ead$ it follows that
  not all {\ab}s are self-adjoint (e.g. $E_\lambda$ with $\lambda \neq 0$ is not) not every {\ab} admits a nondegenerate
  sesquilinear form.
\end{remark}

Consider now a sesquilinear form $\Phi$ on $E$. By applying to it the conjugate functor we obtain a bilinear map
$\check\Phi$ on $\check{E} \times E$ with values in $\check{E_0}$. If we fix an isomorphism of $\check{E_0}$ with $E_0$,
we can consider $\check\Phi$ as a sequilinear form on $\check{E}$. Under this assumption, we define $(a,b)$-hermitian
and anti-$(a,b)$-hermitian forms as:

\begin{definition}
  Let $E$ be an {\ab}. An $(a,b)$-sesquilinear form $H$ on $E$ is called \textbf{$(a,b)$-hermitian} (respectively
  \textbf{anti-$(a,b)$-hermitian}) if it satisfies:
  \[
  H(v,w) = \check{H}(w,v),
  \]
  (respectively
  \[
  \left.H(v,w) = -\breve{H}(w,v)\right).
  \]
  where $v\in E$, $w \in \breve{E}$ and $\breve{H}$ is the sesquilinear form on $\check{E}$ defined above.
\end{definition}

We have already shown that in order to admit a nondegenerate sesquilinear form, an {\ab} must be self-adjoint. We will
refine the concept of self-adjoint by defining:

\begin{definition}
  Let $E$ be a self-adjoint {\ab}. We say that $E$ is \strong{hermitian} (resp.  \strong{anti-hermitian}), if it admits a
  nondegenerate hermitian (resp.  anti-hermitian) form.
\end{definition}

Let $E$ be an {\ab} endowed with an hermitian form and let $\Phi: E \ra \Ead$ be the linear form associated to the
hermitian form via the remark~\ref{rmk:bilinear_linear}.

We can translate the hermitian property into the identity between $\Phi$ and its adjoint $\check{\Phi}^*: E \ra \Ead$.
In fact while $\Phi(v)$, for $v\in E$ is the linear map:
\[
\phi: w \mapsto \Phi(v,w),\qquad w\in \check{E},
\]
the adjoint map $\check{\Phi}^*$ sends an element $v\in E = E^{**}$ to the
map:
\[
\phi: w \mapsto v\left( \check{\Phi}(w, \cdot) \right) = \check{\Phi}(w,v).
\]

We will use this formulation extensively in the following section.

Note moreover that given an isomorphism $\Phi$ from an {\ab} $E$ and its $\delta$-dual $\Ead\otimesab E_\delta$ is
equivalent to specifying an isomorphism between $E\otimesab E_{-\delta/2}$ and
\[
\Ead\otimesab E_\delta \otimesab E_{-\delta/2} \simeq \Ead\otimesab E_{\delta/2}.
\]
Since we have
\[
\widecheck{\left( E\otimesab E_{-\delta/2} \right)}^* \simeq \Ead \otimesab
\check{E}^*_{-\delta/2} \simeq \Ead \otimesab E_{\delta/2},
\]
we can identify an isomorphism of $E$ with its $\delta$-dual with an
hermitian form on $E\otimesab E_{-\delta/2}$.

\section{Existence of hermitian forms}

We will analyze in this section the existence of nondegenerate hermitian forms on regular {\ab}s. We will proceed in two
steps: in the first two subsections we will reduce ourselves to a subclass of {\ab}s called indecomposable {\ab}s and
show that every regular {\ab} can be decomposed into the direct sum of indecomposable ones and that this decomposition
is unique (theorem \ref{krull-schmidt}).

In the last subsection we will show that a self-adjoint {\ab} which is indecomposable admits at least an hermitian or
anti-hermitian form. The result is optimal since there are examples that admit only an hermitian or only an
anti-hermitian form (theorem \ref{symmetric_or_antisymmetric}).

\subsection{Indecomposable {\ab}s}

\begin{definition}
  Let $E$ be an {\ab}. We say that $E$ is \strong{indecomposable} if it cannot be written as direct sum $F\oplus G$ of
  non zero {\ab}s.
\end{definition}

Whenever we decompose an {\ab} $E$ into a direct sum of two {\ab}s $E = F \oplus G$ the rank of the components is
strictly less than the rank of $E$, hence by proceeding by induction for every {\ab} $E$ we can find a decomposition
into a sum of indecomposable {\ab}s:
\[
E = \bigoplus_{i=1}^r F_i,
\]
where $r\in\NN$ and $F_i$ are indecomposable sub-{\ab}s.

We are interested in the question whether the isomorphism classes of the $F_i$ are unique and do not depend upon the
decomposition. We will need to this purpose an introductory result:

\begin{proposition}
  \label{prop:bijective_or_nilpotent}
  Let $E$ be a regular and indecomposable {\ab}. Then every endomorphism of $E$ is either bijective or nilpotent.
\end{proposition}

The proof of this proposition will need several steps beginning with a definition:

\begin{definition}
Let $E$ be a regular {\ab} and $\lambda\in\CC$. We define
\[
V_{\lambda} = \left\{ \sum F_i| F_i\subset E, F_i\simeq
E_{\lambda} \right\}
\]
to be the sum of all sub-{\ab}s of $E$ isomorphic to $E_{\lambda}$.
\end{definition}

The object $V_\lambda$ is clearly a sub-{\ab}.  We will use $V_\lambda$ as an induction step in the proof of
proposition~\ref{prop:bijective_or_nilpotent}, by choosing a $\lambda$ such that $V_\lambda$ is normal:

\begin{proposition}
Let $E$ be a regular \ab, $\lambda\in\CC$ and:
\[
\lambda_{min} = \inf_j\left\{ \lambda + j| \exists x \in E, ax =
(\lambda + j)bx\right\}
\]
be the minimal $\lambda +j$ such that $E$ contains a monome of type
$(\lambda+j, 0)$.

Then $V_{\lambda_{min}}$ is a normal sub-{\ab} of $E$ isomorphic as {\ab} to the direct sum of a finite number of copies
of $E_{\lambda_{min}}$.
\end{proposition}

\begin{proof}
  We will use two facts.

  First, for every $W \simeq \bigoplus E_{\lambda_{min}}$ sub-{\ab} of $E$, $W$ is normal in $E$.  Let in fact $\{e_i\}$
  be a basis of $W$ with $1 \leq i \leq p$ the rank of $W$.  Suppose by absurd that there exist some $x\in W$ which is
  in $bE$, but not in $bW$.

  By eventually translating $x$ by an element of $bW$, we can assume $x= \sum_{i=1}^p \alpha_ie_i$, $\alpha_i\in\CC$. We
  can easily verify that $ax = \lambda_{min}bx$ but now if $x=by$ we must have $ay = (\lambda_{min} - 1) b y$, and since
  $y \in E$ it contradicts the minimality of $\lambda_{min}$.

  On the other hand we can show that $V_{\lambda_{min}}$ is a direct sum of $E_{\lambda_{min}}$. In fact let $W$ be the
  largest (inclusionwise) direct sum of copies of $E_{\lambda_{min}}$ included in $V_{\lambda_{min}}$.  We remark that
  since $W$ is normal, for any sub-{\ab} $F$ isomorphic to $E_{\lambda_{min}}$ only one of two cases is possible: either
\[
W\cap F = \{0\} \text{ or } F\subset W.
\]

  If $W\cap F\neq \{0\}$, let $e$ be the generator of $F$ and $S(b)b^ne\in W$ with $S(0)\neq 0$, then $S(b)e\in W$ by
  normality and $e=S^{-1}(b)S(b)e\in W$. We have therefore $F\subset W$.

  If $W$ contains every sub-{\ab} isomorphic to $E_{\lambda_{min}}$, then it is equal to $V_{\lambda_{min}}$.
  Otherwise there is an $F$ such that $W\cap F = \{0\}$, hence $W\oplus F$ is still in $V_{\lambda_{min}}$, which
  contradicts the maximality of $W$.
\end{proof}

  We will now use the sub-{\ab} $V_{\lambda_{min}}$ to show the following proposition

\begin{proposition}
\label{injective_bijective}
Let $E$ be a regular {\ab} and $\phi$ an $(a,b)$-morphism between $E$ and itself. Then $\phi$ is bijective if and
only if $\phi$ is injective.
\end{proposition}

\begin{proof}
To show that bijectivity follows from injectivity, we will proceed by induction on the rank of the module.

If $E$ is of rank $1$ the statement of the proof is satisfied: in fact $E$ must be isomorphic to one of the $E_\lambda$
and the only $b$-linear morphisms from a $E_\lambda$ to itself that are also $a$-linear are those that send the
generator $e$ to $\alpha e$, $\alpha \in \CC$. They are all bijective for $\alpha\neq 0$.

Let now $E$ be of rank $n>1$. We can find a $\lambda_{min}$ (cf. \cite{barlet1}) that verifies the minimality property
of the previous proposition. Hence the module $V_{\lambda_{min}}$ is normal and isomorphic to a direct sum of copies of
$E_{\lambda_{min}}$.

Let $\{e_i\}$ be a basis of $V_{\lambda_{min}}$ composed of monomials of type $(\lambda_{min}, 0)$ and let $x$ another
monomial of type $(\lambda_{min}, 0)$. We want to show that $x$ is a linear combination of the elements of the basis,
with coefficients in $\CC\subset \CCb$.

From the definition of $V_{\lambda_{min}}$ follows that $x\in V_{\lambda_{min}}$. Suppose now that $x = \sum_i
S_i(b)e_i$ and let us apply $a$ to both sides. We obtain:
\[
ax = \sum_i \left( \lambda_{min}S_i(b)be_i + S_i'(b)b^2e_i\right) = \lambda_{min}bx + \sum_i S_i'(b)b^2e_i
\]
and since $x$ is a monomial of type $(\lambda_{min}, 0)$, we must have $S_i'(b)=0$ for all $i$ and therefore
\[
x = \sum_i S_i(0) e_i,
\]
as we wanted.

Let $\phi:E\ra E$ be an injective endomorphism of $E$ and $\{e_i\}$ a basis of $V_{\lambda_{min}}$. Every $\phi(e_i)$ is
a monomial of type $(\lambda_{min}, 0)$ and therefore is an element of $V_{\lambda_{min}}$.  The restriction of $\phi$
to $V_{\lambda_{min}}$ is therefore an endomorphism of $V_{\lambda_{min}}$:
\[
\phi|_{V_{\lambda_{min}}}: V_{\lambda_{min}} \ra V_{\lambda_{min}}.
\]

Moreover since the coefficients of the $\phi(e_i)$ in our base are complex constants, $\phi|_{V_{\lambda_{min}}}$
behaves as a linear application between finite dimensional spaces: in particular if it is injective, it is also
surjective.

In order to apply our induction hypothesis let us consider the following commutative diagram:
\[
\xymatrix{
0 \ar[r] & V_{\lambda_{min}} \ar@{^{(}->}[r] \ar[d]^\phi & E \ar@{->>}[r] \ar[d]^\phi &
E / V_{\lambda_{min}} \ar[r] \ar[d]^{\tilde\phi} & 0\\
0 \ar[r] & V_{\lambda_{min}} \ar@{^{(}->}[r] & E \ar@{->>}[r] & E / V_{\lambda_{min}} \ar[r] & 0\\
}
\]
where $\tilde\phi$ is the $(a,b)$-linear morphism induced on the quotient.  As we showed the first downward arrow is
bijective

The third arrow $\tilde\phi$ is injective: suppose in fact that we have two different classes with representatives $x$
and~$y\in E$ that map to the same class modulo $V_{\lambda_{min}}$. Then $\phi(x-y)$ is in $V_{\lambda_{min}}$.  From
the bijectivity of $\phi|_{V_{\lambda_{min}}}$ we can find an element $v\in V_{\lambda_{min}}$ such that
$\phi(x-y)=\phi(v)$ which in turn implies $x-y=v$ by the injectivity of $\phi$, which contradicts the fact that $x$ and
$y$ are in distinct classes modulo $V_{\lambda_{min}}$.

Since the rank of $E/V_{\lambda_{min}}$ is strictly inferior to the rank of $E$, we can apply the induction hypothesis
to show that $\tilde\phi$ is also bijective.

It follows from a basic result of homological algebra that the second arrow is bijective if it is injective.
\end{proof}

We can now consider endomorphisms that are not necessarily injective. Once again the structure of {\ab}s does not differ
essentially from that of finite vector spaces over $\CC$:

\begin{lemma}
  \label{lem:ker+im}
Let $E$ be a regular {\ab} and $\phi$ an endomorphism of $E$. Then $E$ splits into the direct sum of two $\phi$-stable
sub-{\ab}s $F$ and~$N$, with $\phi$ bijective on $F$ and nilpotent on $N$.
\end{lemma}

\begin{proof}
Consider the sequence of normal sub-{\ab}s
\[
K_n = \Ker \phi^n,\quad n\in\NN.
\]

Since two normal sub-{\ab}s $F\subset G$ are equal if and only if they have the same rank, the sequence of $K_n$
stabilizes beginning with a certain integer $m$: $K_m = K_{m+1}$.

On the other hand if we consider the sequence $I_n = \im \phi^n$, let us look at the restriction of $\phi$ to $I_m$:
\[
\phi|_{I_m}: I_m \ra I_{m+1}\subset I_m.
\]

This restriction is injective: if $y=\phi^m(x)\in\Ker\phi$, then $x\in K_{m+1}$ which is equal to $K_m$. Hence
$\phi^m(x) = y = 0$.  From the previous proposition we deduce that this restriction is in fact bijective, which means
that $I_{m+1} = \phi(I_m) = I_m$.

We can now take $F = I_m$ and $N = K_m$. They are clearly stable by $\phi$.  We will show that $E = F \oplus N$.

We have in fact $\Ker\phi\cap F = \{0\}$, since the restriction of $\phi$ to $I_m$ is injective. \textit{À fortiori},
since $K\subset \Ker\phi$ we have $F\cap N = \{0\}$.

Let's take an element $x\in E$. Since $I_m = I_{2m}$ we can find an element $y\in E$ such that $\phi^m(x) =
\phi^{2m}(y)$ and call $k$ the element $x - \phi^m(y)$. Thus we can write $x$ as a sum:
\[
x = \phi^m(y) + k
\]
of an element $\phi^m(y)\in I_m$ and an element $k\in K_m$, which implies that:
\[
E = N \oplus F.
\]

The restriction of $\phi$ to $N$ is nilpotent, since $\phi|_N^m = 0$, while we already showed that the restriction to
$I_m= F$ is bijective.
\end{proof}

We have now all the elements necessary to prove proposition~\ref{prop:bijective_or_nilpotent}:

\begin{proof}
  Let $E$ be a regular indecomposable {\ab} and $\phi$ an endomorphism of $E$. Then by lemma~\ref{lem:ker+im} $E$ splits
  into a sum
  \[ E = N \oplus F \]
  of two {\ab}s, with $\phi$ nilpotent on $N$ and bijective on $F$. But $E$ is indecomposable, therefore either $N=0$
  and $\phi$ is bijective or $F=0$ and $\phi$ is nilpotent.
\end{proof}

\subsection{Krull-Schmidt theorem}

This subsection will be devoted to the proof of a version of the Krull-Schmidt theorem for the theory of {\ab}s. The
principal argument of the proof will be proposition~\ref{prop:bijective_or_nilpotent} from the previous subsection.

\begin{theorem}[Krull-Schmidt for {\ab}s]

\label{krull-schmidt}
Suppose that we have two decompositions into direct sum of a regular {\ab} $E$:
\[
E = \bigoplus_{i=1}^m E_i
\]
\[
E = \bigoplus_{i=1}^n F_i
\]
where $m,n\in\NN$ and all $E_i$ and $F_i$ are indecomposable {\ab}s. Then $m=n$ and up to a reindexing of the modules
$E_i$ is isomorphic to $G_i$ for all $1\leq i \leq n$.
\end{theorem}

For the proof of this theorem we need a couple of lemmas:

\begin{lemma}
  Let $E$ be a regular indecomposable {\ab} and $\phi$ an automorphism of $E$. Suppose moreover that $\phi = \phi_1 +
  \phi_2$. Then at least one of $\phi_1$, $\phi_2$ is an isomorphism.
\end{lemma}

\begin{proof}
\label{sum_iso}
Be applying $\phi^{-1}$ to both terms, we can assume without loss of generality that $\phi = Id$ is the identity
automorphism.

The two endomorphisms $\phi_1$ and $\phi_2$ commute. In fact:
\[
\phi_1\phi_2 - \phi_2\phi_1 = \phi_1(\phi_1 + \phi_2) - (\phi_2 + \phi_1)\phi_1 = \phi_1 - \phi_1 = 0.
\]
By lemma \ref{prop:bijective_or_nilpotent} the $\phi_i$ can be either nilpotent or isomorphisms.  If they were both
nilpotent, their sum would be nilpotent, which is absurd. Hence the result.
\end{proof}

\begin{remark}
By subsequently applying the previous lemma, we can extend the result to the sum of more than two endomorphisms.
\end{remark}

\begin{lemma}
\label{composition_iso}
Let $E$ and $F$ be indecomposable regular {\ab}s and $\alpha: E\ra F$ and $\beta: F \ra E$ two $(a,b)$-linear morphisms.
Suppose that $\beta\circ\alpha$ is an isomorphism, then $\alpha$ and $\beta$ are also isomorphisms.
\end{lemma}

\begin{proof}
Let prove that $F = \im \alpha \oplus \Ker \beta$. If $\alpha(x)\in \Ker\beta$, we have
\[
\beta\circ\alpha(x)=0,
\]
hence $x=0$ and therefore
\[
\im\alpha \cap \Ker\beta = \{0\}.
\]

Consider now an element $x\in F$ and let
\[
y = \alpha\circ(\beta\circ \alpha)^{-1}\circ \beta(x).
\]

We have
\[
\beta(x - y) = \beta(x) - \beta(y) = \beta(x) - \left( \beta\circ\alpha
\right)\circ\left( \beta\circ \alpha \right)^{-1} \circ \beta(x) = \beta(x)
- \beta(x) = 0.
\]
We can thus write $x$ as sum of an element $y$ of $\im\alpha$ and an element $x-y$ of $\Ker\beta$. This implies $F =
\im\alpha \oplus \Ker\beta$.

Now since $\beta\circ\alpha$ is injective, so must be $\alpha$ and $\im\alpha$ can not be $0$. But $F$ is indecomposable
therefore we must have $\im\alpha=F$ and $\Ker\beta=0$. It follows that $\alpha$ is bijective and
$\beta=(\beta\circ\alpha)\circ\alpha^{-1}$ must be also bijective.
\end{proof}

\begin{proof}[Proof of Krull-Schmidt theorem for {\ab}s]
We will show this theorem by induction on $m$.

If $m=1$, then $E$ is indecomposable and we must have $n=1$ and $E_1\simeq F_1$.

In the general case consider the morphisms
\[
q_i=\pi_i\circ p_1,
\]
where the $\pi_i$s are the projections on $F_i$ and the $p_j$s are the projections on $E_j$. Let consider the sum:
\[
\sum_i p_1\circ q_i = p_1\circ\sum_i \pi_i\circ p_1=
p_1\circ p_1 = p_1,
\]
is the identity on the component $E_1$. By the lemma~\ref{sum_iso}, there is an $i$ such that $p_1\circ q_i|_{E_1}:E_1
\ra E_1$ is an isomorphism.  Suppose, without loss of generality, it is $p_1\circ q_1$, then by the
lemma~\ref{composition_iso} $q_1|_{E_1}=\pi_1:E_1 \ra F_1$ is an isomorphism.

In order to apply the induction hypothesis, let note $G=\sum_{i=2}^m F_i$.  We want to show that $E_1 \oplus G$ is equal
to $E=F_1 \oplus G$. Since $\pi_1$ is an isomorphism of $E_1$ onto $F_1$ and its kernel is $G$ we must have
\[
E_1 \cap G = \{0\}:
\]
if $x \in E_1 \cap G$, then $\pi_1(x) = 0$, but $\pi_1$ restricted to $E_1$ is injective, so $x=0$. On the other hand
every element of $E$ can be written as $v + w$ with $v\in F_1$ and $w\in G$. If $y\in E_1$ is such that $\pi_1(y)=v$,
then we have:
\[
v + w = y + \pi_1(y) - y + w,
\]
and $\pi_1(y) - y\in W$ by definition of $\pi_1$. We can then conclude that $E= E_1 \oplus G=E_1 +\sum_{i=2}^m E_i$.

We have immediately $E/E_1 \simeq G \simeq \sum_{i=2}^m E_i$ and we can apply the induction hypothesis to $G$.
\end{proof}

We can now focus on finding hermitian isomorphisms of an {\ab} $E$ with its adjoint $\Ead$. The Krull-Schmidt theorem
will be useful to show the following decomposition:

\begin{proposition}
  \label{prop:decomposition of self-adjoint modules}
  Let $E$ be a regular self-adjoint {\ab}. Then $E$ is isomorphic to:
\[
E \simeq \bigoplus_{i=1}^r \left( F_i^{\oplus \alpha_i} \right) \oplus \bigoplus_{i=1}^s \left( G_i \oplus \check{G}^*_i\right)^{\oplus \beta_i}
\]
where $r$ and $s$ as well as the $\alpha_i$ and $\beta_i$ are positive integers. The $F_i$ are self-adjoint {\ab}s and
the $G_i$ are non self-adjoint {\ab}s. The isomorphism classes of the $F_i$, $G_i$ and $\check{G}^*_i$ are all disjoint.
\end{proposition}

\begin{proof}
Consider a decomposition of $E$ into indecomposable {\ab}s
\[
E = \sum_i E_i.
\]
Since $E$ is self-adjoint we have another decomposition given by
\[
E \simeq \check{E}^* = \sum_i \check{E}^*_i.
\]
The Krull-Schmidt theorem assures us that the factors are unique up to a permutation. So we can divide the $E_i$ into
two groups.

In the first group we find the self-adjoint components $F_i$ with a certain multiplicity.

In the second one we find the non self-adjoint components $G_i$ with the respective multiplicity. Since the two
decompositions $\sum_i E_i$ and $\sum_i \check{E}^*_i$ must contains the same modules up to a permutation, the
multiplicity of the $G_i$ and the $\check{G}_i^*$ must be equal.
\end{proof}

\begin{remark}
From the definition above we can immediately see that the non self-adjoint part of the decomposition always admits a
hermitian nondegenerate form In fact if we consider the module $G_i \oplus \check{G}_i^*$, a hermitian form can be given
by:
\begin{align*}
\Phi:G_i \oplus \check{G}_i^* & \ra (G_i\oplus \check{G}_i^*)^*\check{}
= \check{G}_i^* \oplus G_i\\
 (x,y) & \mapsto (y,x).
\end{align*}
If the multiplicity of a self-adjoint term $F_i$ is pair, we fall into the
same situation.

The case of an unpair multiplicity of a self-adjoint component is far more interesting and we will study it in the next
subsection.
\end{remark}

\subsection[Hermitian forms on indecomposable {\ab}s]{Hermitian forms on indecomposable {\ab}s}

As already noted in the previous subsection, the situation of an indecomposable self-adjoint {\ab} concerning hermitian
forms is far less regular and the existence is not always guaranteed. We have in fact the following theorem:
\begin{theorem}
\label{symmetric_or_antisymmetric}
Let $E$ be a regular indecomposable self-adjoint {\ab} and $E\neq\{0\}$. Then it admits a hermitian nondegenerate form or
an anti-hermitian one.
\end{theorem}

\begin{proof}
Let $\Phi:E\ra \check{E}^*$ be any isomorphism of $E$ with its dual and pose $M=\Phi^{-1}\check{\Phi}^*$. Consider now
the two endomorphisms of $E$ given by:
\[
Id + M
\]
and
\[
Id - M
\]
they commute and can be either isomorphisms or nilpotent, since $E$ is indecomposable.  But if they were both nilpotent,
their sum $2Id$ would be nilpotent too, which is absurd.

If $Id + M$ is an isomorphism, so is $S = \Phi + \check{\Phi}^*$, which is associated to a nondegenerate hermitian form.
The bijectivity of $Id - M$ on the other hand gives us an isomorphism $A = \Phi - \check{\Phi}^*$, which comes from an
anti-hermitian form.
\end{proof}

Note that all the cases of the previous theorem are equally possible.

\begin{example}
The simplest example of a regular self-adjoint and indecomposable {\ab} which admits only a hermitian form is the elementary
{\ab} $E_0$ with the isomorphism that sends the generator $e$ to its adjoint $\check{e}^*$.
\end{example}

\begin{example}
In order to obtain only an anti-hermitian form, we can consider for a given $\lambda,\mu\in\CC$ the {\ab} $E$ of rank
$4$, generated by $\{e_1, e_2, e_3, e_4\}$ which verifies:
\begin{align}
ae_1 &= \lambda be_1\notag\\
ae_2 &= \mu be_2 + e_1\notag\\
  \label{eqn:antihermitian}
ae_3 &= -\mu be_3 + e_1\\
ae_4 &= -\lambda be_4 + e_2 - e_3\notag
\end{align}
whose adjoint basis satisfies:
\begin{align*}
a\cdot\check{e}_4^* &= \lambda b\check{e}_4^*\\
a\cdot\check{e}_3^* &= \mu b\check{e}_3^* - \check{e}_4^*\\
a\cdot\check{e}_2^* &= -\mu b\check{e}_2^* + \check{e}_4^*\\
a\cdot\check{e}_1^* &= -\lambda b\check{e}_1^* + \check{e}_3^* + \check{e}_2^*.
\end{align*}
It is easy to show by calculation that the only isomorphism between $E$ and $\Ead$ is, up to mutliplication by a complex
number, the one that sends $e_1$, $e_2$, $e_3$ and $e_4$ into $\check{e}_4$, $-\check{e}_3$, $\check{e}_2$ and
$-\check{e}_1$ respectively.

This is isomorphism is anti-hermitian and since there are no other isomorphisms $E$ is also indecomposable.
\end{example}

\begin{example}
  The regular {\ab} $E_0 \oplus E_0$ admits both an hermitian and anti-hermitian form.
\end{example}

\section{Duality of geometric $(a,b)$-modules}

In the study of the Brieskorn lattice K.~Saito introduced the concept of ``higher residue pairings'' (cf. \cite{saito}),
which can be defined using a set of axiomatic properties. 

Using the theory of $(a,b)$-modules R.~Belgrade showed the existence of a duality isomorphism between an $(a,b)$-module
associated to a germ of a holomorphic function in $\CC^{n+1}$ with an isolated singularity at the origin and its
$(n+1)$-dual. In this section we'll prove (as already noticed by R.~Belgrade in \cite{belgrade}) that the concept of
``higher residue pairings'' and self-adjoint {\ab} are linked.

In this section $D$ will always denote the Brieskorn module associated to a holomorphic function in $\CC^{n+1}$ with an
isolated singularity, while $E$ will denote its $b$-adic completion considered as an {\ab}.

The following theorem of R.~Belgrade gives a relationship between $E$ and its $(n+1)$-dual.
\begin{theorem}[Belgrade]
\label{letheoreme}
Let $E$ be the $(a,b)$-module associated to a germ of holomorphic function $f: \CC^{n+1} \rightarrow \CC$, then there is
a natural isomorphism between $E$ and its $(n+1)$-dual:
\[
\Delta:\quad E \simeq \breve{E}^*\otimesab E_{n+1}
\]
\end{theorem}

We can obtain from this isomorphism a series $\Delta_k: E \times E \rightarrow \CC$ of bilinear forms defined as follow:
\[
\left[\Delta(y)\right](x) = (n+1)!\sum_{k=0}^{+\infty} \Delta_k (x, y)
b^ke_{n+1}
\]
with $x$, $y \in E$.

\section{``Higher residue pairings'' of K. Saito}

K. Saito introduced in \cite{saito} a series of pairings on the Brieskorn
lattice $D$ which are called ``higher residue pairings'':
\[ K^{(k)}:\quad D\times D \rightarrow \CC \qquad k\in \mathbb{N}\]
which are characterized by the following properties:
\begin{enumerate}
\item $K^{(k)}(\omega_1,\omega_2) = K^{(k+1)}(b\omega_1,\omega_2) = -
K^{(k+1)}(\omega_1,b\omega_2)$.
\item $K^{(k)}(a\omega_1,\omega_2) - K^{(k)}(\omega_1,a\omega_2) =
(n+k)K^{(k-1)}(\omega_1,\omega_2)$.
\item $K^{(0)}$ satisfies
  \[K^{(0)}(D,bD) = K^{(0)}(bD,D) = 0\]
  and induces Grothendieck's residue on the quotient $D/bD$.
\item $K^{(k)}$ are $(-1)^k$-symmetric.
\end{enumerate}

\begin{remark}
We notice that from the properties (i) and (iii) above we can deduce that
$K^{(k)}(D,b^{k+1}D) = K^{(k)}(b^{k+1}D,D) = 0$, so we can consider the
pairings $K^{(k)}$ as being defined on $D/b^{k+1}D$.
\end{remark}

In the following section we'll show the following result:

\begin{proposition}
The $\Delta_k$ verify the properties (i)--(iii) of the ``higher residue
pairings'' of K.~Saito.
\end{proposition}

The prove will be performed by steps.

\section{Proof of the proposition}

\subsection{Proof of (i)}
We use the $b$-linearity of $\Delta(y)$ to obtain:
\begin{multline*}
\sum_k(n+1)!\Delta_k(bx,y)b^ke_{n+1} = \left[\Delta(y)\right](bx) =
b \left[\Delta(y)\right](x) =\\
\sum_k(n+1)!\Delta_k(x,y)b^{k+1}e_{n+1}
\end{multline*}
which gives us $\Delta_k(x,y)=\Delta_{k+1}(bx,y)$. And similarly by using
the $b$-linearity of $\Delta$ and the adjoint morphism, we obtain:
\[
\Delta(by)(x) = \check{\Delta}^*(x)(by) =
-b\check{\Delta}^*(x)(y)=-b\Delta(y)(x),
\]
and therefore
\begin{multline*}
(n+1)!\sum_k\Delta_k(x,by)b^ke_{n+1} = \Delta(by)(x) =
-b \Delta(y)(x) =\\
(n+1)!\sum_k-\Delta_k(x,y)b^{k+1}e_{n+1},
\end{multline*}
which implies $\Delta_k(bx,y) = -\Delta_{k+1}(x,by)$.

\subsection{Proof of (ii)}
Since $\Delta$ is an isomorphism we have $\Delta(ay) =
a\cdot_{\check{E}^*\otimes E_\delta}[\Delta(y)]$ and:
\[(n+1)!\sum_k\Delta_k(x,ay)b^ke_{n+1} = \Delta(ay)(x) = a\cdot[\Delta(y)](x) =
\]\[
= \Delta(y)(ax) - a[\Delta(y)(x)] =
(n+1)!\sum_k\left(\Delta_k(ax,y)b^ke_{n+1} -
\Delta_k(x,y)ab^ke_{n+1}\right)\]
The definition of $(a,b)$-module and $E_{n+1}$
($ae_{n+1}=(n+1)be_{n+1}$) gives the following relation
\[ ab^ke_{n+1} = b^kae_{n+1} + kb^{k+1}e_{n+1} = (n+k+1)b^{k+1}e_{n+1} \]
hence follows:
\[
\Delta_k(ax,y) - \Delta_k(x,ay) = (n+k)\Delta_{k-1}(x,y)
\]

\subsection{Grothendieck's residue}

We have to show now that the pairing $\Delta_0$ induces Grothendieck's
residue on $D/bD\simeq \Omega^{n+1}/df \wedge \Omega^n$.

\textbf{Proof of (iv):} From the definition of $\Delta_0$ and the
$b$-linearity of $\Delta$ it's easy to see that $\Delta_0(D,bD) =
\Delta_0(bD, D) = 0$. We can hence consider $\Delta_0$ as a pairing on
$D/bD$.

Grothendieck's residue is defined as follows:
\[
Res(g,h) := \lim_{\varepsilon_j \rightarrow 0,\forall j}
\int_{|\dd f / \dd z_j| = \varepsilon_j}
\frac{gh \dz}{\dd f/ \dd z_1 \cdots \dd f/ \dd z_{n+1}}
\]
where $g,h \in \mathcal{O}$ and $\dz = \dz_1 \wedge \dots \dz_{n+1}$.

\noindent The morphism $\Delta$ is defined as composed morphism of six \hbox{$(a,b)$-modules} morphism (\cite{belgrade})
as showed by the following graph:
\[
\xymatrix
{
E \ar[r]^\alpha & F_1 \ar[r]^\beta & F_2 & \\
& & & F_3 \ar[ul]^\gamma \ar[dl]^\delta\\
\breve{E}_{n+1} & F_5 \ar[l]^\zeta & F_4 \ar[l]^\eta & \\
}
\]
These morphisms pass to the quotient by the action of $b$ in order to give
a decomposition of the morphism $\Delta_0$:
\[
\xymatrix
{
E/bE \ar[r]^{\tilde\alpha} & F_1/bF_1 \ar[r]^{\tilde\beta} & F_2/bF_2 & \\
& & & F_3/bF_3 \ar[ul]^{\tilde\gamma} \ar[dl]^{\tilde\delta}\\
\frac{\left(\breve{E}^*\otimes E_{n+1}\right)}{b\left(\breve{E}^*\otimes
E_{n+1}\right)} & F_5/bF_5 \ar[l]^{\tilde\zeta} & F_4/bF_4
\ar[l]^{\tilde\eta} & \\
}
\]
We have to verify that the image of $[g\dz]$ by $\Delta_0$ is
$Res(g,\cdot)$, where $g\dz$ is an element of $\Omega^{n+1}$. We'll
accomplish this in many steps using the decomposition above.

\begin{enumerate}
\item {\bf Step 1: $E$, $F_1$ and $F_2$.} We have the following
isomorphisms:
\[
\frac{F_1}{bF_1} \simeq \frac{\Omega^{n+1}}{df\wedge \Omega^n}\qquad
\frac{F_2}{bF_2} \simeq \frac{\Db^{n+1}}{(\dbar - df\wedge) \Db^n},
\]
the morphism $\tilde\alpha$ coincides with the identity on
$\Omega^{n+1}/df\wedge \Omega^n$ and $\tilde\beta$ is induced by the
inclusion $i: \Omega^{n+1} \mapsto \Db^{n+1}$. We deduce that
$\tilde\beta\circ\tilde\alpha([g\dz])=[i(g\dz)]$. Let write
$T\in\Db^{n+1,0}$ the current $i(g\dz)$.

\item {\bf Step 2: path between $F_2$ and $F_3$} By using the description
of the lemma 3.4.2 of \cite{belgrade} we see that:
\[
\frac{F_3}{bF_3} = \frac{\Ker(\Db^{0,n+1} \stackrel{df\wedge}{\rightarrow}
\Db^{1,n+1})}{\dbar\Ker(\Db^{0,n} \stackrel{df\wedge}{\rightarrow}
\Db^{1,n})}
\]
and the isomorphism $\tilde\gamma$ is induced by the inclusion
$\Db^{0,n+1}\subset \Db^{n+1}$. In order to find $S:=\tilde\gamma^{-1}(T)$
we have to solve the following system:
\begin{eqnarray*}
T & = & df \wedge \alpha^{n,0} \\
\dbar \alpha^{n,0} & = & df \wedge \alpha^{n-1,1} \\
\cdots && \cdots \\
\dbar \alpha^{1, n - 1} & = & df \wedge \alpha^{0,n}\\
\dbar \alpha^{0,n} & = & S\\
\end{eqnarray*}
where the $\alpha^{p,q} \in \Db^{p,q}$. There is a solution to this system
of equations since the complex $(\Db^{\bullet,q};df\wedge)$ is acyclic in degree
$\neq 0$ for all $q$ in $0,\ldots,n+1$ and the solution verifies $[S] =
[T]$ where $[\cdot]$ is the equivalence class in $F_2/bF_2$.
\[
(\dbar - df\wedge) \sum_{k=0}^n \alpha^{k,n-k} = \dbar
\alpha^{0,n} - df \wedge \alpha^{n,0} = S - T
\]

We can compute this solution explicitly.  Let be $(p,q)\in \mathbb{N}^2$
and $\phi^{p,q}$ a $C^{\infty}$ test form with compact support and of type
$(p,q)$. The action of $T$ over $\phi^{0,n+1}$ is given by:
\[ <T,\varphi^{0,n+1}> = \int \varphi^{0,n+1}\wedge g\dz \]
then the following current verifies $T= df \wedge \alpha^{n,0}$:
\[ <\alpha^{n,0}, \varphi^{1,n+1}> = \lim_{\varepsilon_1 \rightarrow 0}
\int_{|\dd_1f|\geq\epsilon_1} \frac{\varphi^{1,n+1}\wedge g\dz_2\wedge \ldots
\wedge \dz_{n+1}}{\dd_1f} \]
in fact:
\begin{eqnarray*}
<df\wedge \alpha^{n,0}, \varphi^{0,n+1}> &=& \lim_{\varepsilon_1 \rightarrow 0} \int_{|\dd_1f|\geq\epsilon_1}
\frac{\varphi^{0,n+1}\wedge df \wedge g\dz_2\wedge \ldots \wedge \dz_{n+1}}{\dd_1f} \\
&=& \int \varphi^{0,n+1}\wedge g\dz\\
\end{eqnarray*}
and thanks to the Stokes' theorem:
\begin{eqnarray*}
<\dbar\alpha^{n,0},\varphi^{1,n}>&=&-<\alpha^{n,0},\dbar\varphi^{1,n}>\\ &=&\lim_{\varepsilon_1 \rightarrow 0} -
\int_{|\dd_1 f| \geq \epsilon_1} \frac{\dbar\varphi^{1,n}\wedge g\dz_2\wedge \ldots \wedge \dz_{n+1}}{\dd_1f}\\
&=& \lim_{\varepsilon_1 \rightarrow 0} \int_{|\dd_1 f| = \epsilon_1} \frac{\varphi^{1,n}\wedge g\dz_2\wedge \ldots
\wedge \dz_{n+1}}{\dd_1f}\\
\end{eqnarray*}

\noindent We'll remark that the currents $\alpha_k^{n,0}$ defined below for $1 \leq k \leq n+1$ also satisfy $df \wedge
\alpha^{n,0}_k = T$:
\[
<\alpha_k^{n,0},\varphi^{1,n+1}> =
\]
\[ =
\lim_{\varepsilon_k \rightarrow 0}
\int_{|\dd_kf|\geq\epsilon_k} \frac{(-1)^{k+1}\varphi^{1,n+1}\wedge
g\dz_1\wedge \ldots \widehat{\dz_k} \ldots \wedge \dz_{n+1}}{\dd_kf}
\]
and that $[\dbar\alpha^{n,0}] = [\dbar\alpha_k^{n,0}]$ in $F_2/bF_2$: in
fact $(\dbar - df \wedge) (\alpha^{n,0} - \alpha_k^{n,0}) =
\dbar\alpha^{n,0} - \dbar\alpha_k^{n,0}$.

For all $k \in 0,\ldots,n$ and $1\leq i_1 < \ldots < i_{k+1} \leq n+1$ let us
define:
\[
\alpha_{i_1,\ldots,i_{k+1}}^{n-k,k} =
\frac{1}{(k+1)!}\mathop{\lim_{\epsilon_{i_q} \rightarrow 0}}_{\forall 1\leq
q \leq k+1} \mathop{\int_{|\dd_{i_1}f| \geq \epsilon_{i_1}}}_{|\dd_{i_q}f|
= \epsilon_{i_q}} \frac{(-1)^{(\sum_q i_q) + 1}g\bigwedge_{l\neq
i_1,\ldots,i_{k+1}}\dz_l}{\dd_{i_1}f\ldots\dd_{i_{k+1}}f}
\]
and let $\alpha^{n-k,k} := \alpha_{1,2,\ldots,k+1}^{n-k,k}$.

A simple computation gives us:
\[
\left<df\wedge\alpha_{i_1,\ldots,i_{k+1}}^{n-k,k},\phi^{k,n-k+1}\right> =
\left<\frac{1}{k+1} \sum_{q=1}^{k+1} \dbar \alpha_{i_1, \ldots , \widehat{i_q}
,\ldots ,i_{k+1}}^{n-k+1,k-1}, \phi^{k,n-k+1}\right>
\]
using this formula, we can prove by induction on $k$ that the class of the current $\alpha_{i_1,\ldots,i_{k+1}}^{n-k,k}$
doesn't depend upon the $i_q$s. This gives us
\[
\left[df\wedge \alpha^{n-k,k}\right] =
\left[\dbar\alpha^{n-k+1, k-1}\right].
\]

\noindent In particular
$\dbar\alpha^{0,n}$ acts upon the test function $\phi^{n+1,0}$
in the following way:
\[ <\dbar\alpha^{0,n},\varphi^{n+1,0}> = \frac{1}{(n+1)!}\lim_{\stackrel{\epsilon_k
\rightarrow 0}{\forall k}} \int_{\stackrel{|\dd_kf| = \epsilon_k}{\forall
k}}
\frac{\varphi^{n+1,0}g}{\dd_1 f \ldots \dd_{n+1} f}
\]

\item {\bf Step 3 from $F_3/bF_3$ to $(D/bD)^{*}$}: let notice that
$S$ is a current of type $(0,n+1)$ with support in the origin.

\noindent We have the following isomorphisms:
\[
\frac{F_4}{bF_4} \simeq
\Ker\left(\mathcal{H}_0^{n+1}(X,\mathcal{O})\stackrel{df\wedge}\rightarrow
\mathcal{H}_0^{n+1}(X,\Omega^1)\right)
\]
and the isomorphism between $F_3/bF_3$ and $F_4/bF_4$ is the natural one,
and
\[
\frac{F_5}{bF_5} \simeq \left(\frac{\Omega^{n+1}}{df\wedge\Omega^n}\right)^{*}
\]

\end{enumerate}

From steps 1--3 we deduce that $\Delta_0$ induces Grothendieck's
residue.

\subsection{Property (iv)}
We will that the isomorphism given by R.~Belgrade can be easily transformed
into one that verifies the property.

Let $\Delta: E \ra \Ead \otimesab E_{n+1}$ be Belgrade's isomorphism. By tensoring with $E_{(n+1)/2)}$ we can show that,
the isomorphisms between $E$ and $\check{E}^*\otimesab E_{n+1}$ are in bijection with the isomorphisms between $E
\otimesab E_{-(n+1)/2}$ and it adjoint, through the map that sends an isomorphism $\Phi$ to $\Phi\otimesab
Id_{E_{-(n+1)/2}}$.

By an easy calculation we can prove the following lemma:

\begin{lemma}
  Let $\Delta: E \ra \Ead \otimes E_{n+1}$ be an isomorphism and
  \[
  \Delta(y)(x) = (n+1)!\sum_k \Delta_k(x,y)b^ke_{n+1}
  \]
  for each $x$ and~$y\in E$.  Then the $\Delta_k$ satisfy Saito's condition (iv) if and only if the isomorphism $\Delta
  \otimesab Id_{E_{-(n+1)/2}}$ is hermitian.
\end{lemma}

\begin{proof}
  $\Delta \otimesab Id_{E_{-(n+1)/2}}$ is self-adjoint iff we have:
  \begin{multline*}
  \Delta \otimesab Id_{E_{-(n+1)/2}} \left( y \otimes e_{-(n+1)/2}
  \right)\left( x \otimes e_{-(n+1)/2} \right) = \sum_k S_kb^ke_0
  \Leftrightarrow \\
  \Delta \otimesab Id_{E_{-(n+1)/2}} \left( x \otimes e_{-(n+1)/2}
  \right)\left( y \otimes e_{-(n+1)/2} \right) = \sum_k S_k(-b)^ke_0.
\end{multline*}
for all $x$ and~$y\in E$. On the other hand we have:
\begin{multline*}
\Delta \otimesab Id_{E_{-(n+1)/2}} \left( y \otimes e_{-(n+1)/2}
  \right)\left( x \otimes e_{-(n+1)/2} \right) = \sum_k S_kb^ke_0
  \Leftrightarrow \\
  \Delta(y)(x) = \sum_k S_kb^ke_{n+1}.
\end{multline*}
\end{proof}

By combining the previous equivalence with the results on the existence of hermitian forms, we can extend Belgrade's
result:

\begin{theorem}
  Let $E$ be a regular {\ab} associated to a holomorphic function from
  $\CC^{n+1}$ to $\CC$ with an isolated singularity. Then there exists an
  isomorphism $\Phi: E \ra \Ead \otimesab E_{n+1}$ with
  \[
  \Phi(y)(x) = (n+1)!\sum_k \Phi_k(x,y)b^ke_{n+1},
  \]
  for all $x$ and~$y$ such that the sequence of $\CC$-bilinear forms
  $\Phi_k$ satisfies all four properties of Saito's ``higher residue
  pairings''.
\end{theorem}
  
\begin{proof}
  Let $\Delta$ be Belgrade's isomorphism and $\Delta_k$ defined as at the
  beginning of this section. Consider the isomorphism
  \[
  \check{\Delta}^* \otimesab Id_{E_{n+1}}: E \ra \Ead \otimesab E_{n+1}
  \]
  and let $\Phi = \left( \Delta + \check{\Delta}^* \otimesab
  Id_{E_{n+1}}\right)/2$.

  It is easy to see that the $\Phi_k$ satisfy properties (i) and~(ii).  Moreover since $\Delta_0$ is symmetric
  (Grothendieck's residue) and $\check{\Delta}^* \otimesab Id_{E_{n+1}}$ induces the transposed of $\Delta_0$ on $E/bE$,
  we have
  \[
  \Phi_0 = \left(\Delta_0 + {}^t\Delta_0\right)/2 = \Delta_0.
  \]

  We have also
  \begin{multline*}
  \Phi\otimesab Id_{E_{-(n+1)/2}} = \widecheck{\left(\Phi \otimesab
  Id_{E_{-(n+1)}}\right)}^* = \check{\Phi}^* \otimesab Id_{E_{(n+1)/2}} =\\ \Phi
  \otimesab Id_{E_{-(n+1)/2}},
\end{multline*}
  therefore the $\Phi_k$ satisfy Saito's property (iv).

  We just have to show that $\Phi\otimesab Id_{E_{-(n+1)/2}}$ is an isomorphism. Since there exists an isomorphism
  between $E \otimesab E_{-(n+1)/2}$ and its adjoint  $\Delta\otimesab Id_{E_{-(n+1)/2}}$ , we can apply
  proposition~\ref{injective_bijective} and reduce ourselves to prove the injectivity of $\Phi\otimesab
  Id_{E_{-(n+1)/2}}$. But if $\Phi\otimesab Id_{E_{-(n+1)/2}}$ were not injective $\Phi$ would induce a degenerate form
  on $E/bE$, which is absurd.
\end{proof}

The existence of a hermitian form on $E\otimesab E_{-(n+1)/2}$ gives us an interesting restriction on the kind of {\ab}
associated with Brieskorn lattices:

\begin{corollary}
  Let $E$ be a regular {\ab} associated to a holomorphic function from $\CC^{n+1}$ to $\CC$ with an isolated
  singularity. Then $E\otimesab E_{-(n+1)/2}$ is a hermitian {\ab}.
\end{corollary}

\section{Acknowledgements}

I would like to thank Daniel \textsc{Barlet} for his guidance during my
Ph.D.\ thesis, Michel \textsc{Meo} for his help on the complex analysis
topics.

\nocite{barlet1}
\nocite{barlet2}
\nocite{belgrade}
\nocite{brieskorn}
\nocite{deligne}
\nocite{saito}
\nocite{agv}
\nocite{geandier}
\bibliographystyle{alpha}
\bibliography{hermitian}
\end{document}